\UseRawInputEncoding
\documentclass[12pt, reqno]{amsart}
\usepackage{amssymb,latexsym,amsmath,amscd,amsfonts}
\usepackage{latexsym}
\usepackage[mathscr]{eucal}
\usepackage{bm}
\usepackage{mathptmx}

\voffset = -28pt \hoffset = -54pt \textwidth = 6.3in \textheight = 9.5in \numberwithin{equation}{section}

\newcommand{\beg}{\begin{equation}}
\newcommand{\eeg}{\end{equation}}
\newcommand{\ben}{\begin{eqnarray*}}
	\newcommand{\een}{\end{eqnarray*}}

\newtheorem{thm}{Theorem}[section]
\newtheorem{cor}[thm]{Corollary}
\newtheorem{lem}[thm]{Lemma}

\newtheorem{prop}[thm]{Proposition}
\numberwithin{equation}{section} 
\theoremstyle{definition}
\newtheorem{defn}[thm]{Definition}

\newtheorem{eg}[thm]{Example}

\newcommand{\HS}{\mathcal H}

\newcommand{\D}{\mathbb{D}}

\newcommand{\A}{\rm{Aut}}
\newcommand{\ov}{\overline}

\begin{document}
	\title[On $q$-commuting co-extensions and $q$-commutant lifting]
	{On $q$-commuting co-extensions and $q$-commutant lifting}
	
	\author[Bisai, Pal and Sahasrabuddhe]{Bappa Bisai, Sourav Pal and Prajakta Sahasrabuddhe}
	\address[Bappa Bisai]{Stat-Math Unit, Indian Statistical Institute Kolkata, 203 B. T. Road, Kolkata - 700108, India.}
	\email{bappa.bisai1234@gmail.com}
	\address[Sourav Pal]{Mathematics Department, Indian Institute of Technology Bombay,
		Powai, Mumbai - 400076, India.} \email{sourav@math.iitb.ac.in}
	\address[Prajakta Sahasrabuddhe]{Mathematics Department, Indian Institute of Technology Bombay,
		Powai, Mumbai - 400076, India.} \email{prajakta@math.iitb.ac.in}

	
	
	\thanks{The first named author has been supported by the Ph.D Fellowship of University Grand Commissions (UGC), India and the Visiting Scientist Fellowship of Indian Statistical Institute, Kolktata. The second named author
is supported by the Seed Grant of IIT Bombay, the CPDA and the
MATRICS Award (Award No. MTR/2019/001010) of
Science and Engineering Research Board (SERB), India. The third named author has been supported by the Ph.D Fellowship of Council of Scientific and Industrial Research (CSIR), India.}

	\begin{abstract}
	
Consider a nonzero contraction $T$ and a bounded operator $X$ satisfying $TX=qXT$ for a complex number $q$. There are some interesting results in the literature on $q$-commuting dilation and $q$-commutant lifting of such pair $(T,X)$ when $|q|=1$. Here we improve a few of them to the class of scalars $q$ satisfying $|q|\leq \dfrac{1}{\|T\|}$.

	\end{abstract}
	
	\maketitle
	
	\section{Introduction}
	
	\vspace{0.4cm}

	\noindent Throughout the paper we consider only bounded operators acting on complex Hilbert spaces. A contraction is an operator with norm not greater than $1$. The aim of this paper is to contribute to the study of dilation and lifting of $q$-commuting and $q$-intertwining operators.
	
	\begin{defn}
		For a complex number $q$, a pair of operators $(T_1, T_2)$ acting on a Hilbert sapce $\mathcal{H}$ is said to be $q$-\textit{commuting} if $T_1T_2 = q T_2T_1$. Also, for $T_1 \in \mathcal B(\HS_1)$ and $T_2 \in \mathcal B(\HS_2)$, the pair $(T_1,T_2)$ is said to be $q$-\textit{intertwining} by an operator $A\in \mathcal B(\HS_1,\HS_2)$ if $AT_1=qT_2A$.
	\end{defn}
		
	A nontrivial step in operator theory is the isometric (or unitary) dilation of a contraction due to Sz.-Nagy, \cite{Bela} which states the following: for any contraction $T$ acting on a Hilbert space $\HS$, there is a Hilbert space $\mathcal K \supseteq \HS$ and an isometry (or a unitary) $V$ on $\mathcal K$ such that $T^n=P_{\mathcal{H}}V^n|_{\mathcal{H}}$ for all $n \geq 0$ and that the dilation is minimal in the sense that
	\[
	\mathcal{K}=\bigvee\limits_{n=0}^{\infty}V^n\mathcal{H}=\ov{span} \, \{ V^nh\,:\, h \in \HS,\; n \geq 0 \}.
	\]
	Moreover, such a minimal dilation is unique upto isomorphism and $V^*$ is the minimal co-isometric extension of $T^*$, that is, $\HS$ is invariant under $V^*$ and $V^*|_{\HS}=T^*$. Thus, $(V, \mathcal K)$ is the minimal isometric lift of $(T, \HS)$. Ando, \cite{Ando} extended Sz.-Nagy's result to a pair of commuting contractions $T_1,T_2$ acting on $\HS$. Indeed, he showed that there are commuting isometries $V_1,V_2$ on $\mathcal K \supseteq \HS$ such that $T_1^{n_1}T_2^{n_2}=P_{\mathcal{H}}V_1^{n_1}V_2^{n_2}|_{\mathcal{H}}\;$ for all $n_1,n_2\geq 0$. However, no further generalization was possible as was proved by S. Parrott in \cite{Parrot} via a counter example showing that a triple of commuting contractions $(T_1,T_2,T_3)$ may not dilate to a commuting triple of isometries $(V_1,V_2,V_3)$.	Sarason, \cite{sarason} showed that if the minimal co-isometric extension $V$ of $T$ is such that $V^*$ is a unilateral shift of multiplicity one, then any commutant $X$ of $T$ has a norm preserving extension $ Y $ such that $YV=VY$. Sz.-Nagy and Foias generalized the seminal work of Sarason for an arbitrary contraction $T$ which is known as the classical commutant lifting theorem.
	\begin{thm} \label{Classic:Comm}
	Let $T$ be a contraction on a Hilbert space $\HS$ and let $(V, \mathcal K)$ be the minimal isometric (or minimal unitary) dilation of $T$. If $R$ commutes with $T$, then there exists an operator $S$ commuting with $V$ such that $\|R\|=\|S\|$ and $RT^n=P_{\HS}SV^n|_{\HS}$ for all $n \geq 0$.
	\end{thm}
Sz.-Nagy and Foias, \cite{nagy1968dilatation, sz1971lifting} proved a variant of this theorem replacing the minimal isometric dilation by the minimal co-isometric extension of $T$, which further was established independently by Douglas, Muhly and Pearcy in \cite{Douglas}. After a few decades, Sebesty$\acute{e}$n generalized the classical commutant lifting theorem to the $q$-commuting setup for $q=\pm 1$ in the following way.

	\begin{thm}[Theorems 2 \& 3, \cite{Sebestyen}] \label{Sebestn}
		Let $q=1$ or $-1$, $T\in \mathcal{B}(\mathcal{H})$ be any contraction and $X$ be an operator on $\mathcal{{H}}$ satisfying $TX=qXT$. If $V$ acting on a Hilbert space $\mathcal{K} \supseteq \HS$ is the minimal unitary (or isometric) dilation of $T$, then there exists an operator $Y_q$ on $\mathcal{K}$ such that $Y_q$ dilates $X$, $\|Y_q\|=\|X\|$ and $VY_q=qY_qV$.
	\end{thm}
	
	Recently Keshari and Mallick, \cite{K.M.} improved Sebesty$\acute{e}$n's result for $q$-commuting and $q$-intertwining operators with $|q|=1$. Also, Mallick and Sumesh have established a more general operator theoretic version of the same results in \cite{sumesh}. \\
	
	The main results of this paper, Theorems \ref{Dmain}, \ref{qcoisoext}, \ref{main}, \ref{main112} show that these existing results can be further generalized for any complex number $q$ satisfying $|q| \leq 1/{\|T\|}$. Needless to mention that it includes all scalars $q$ from the closed unit disk $\ov{\D}$. We also provide independent proofs to some existing results in this direction. 
 	
\vspace{0.3cm}

	\section{Examples and preparatory results}
	
	\vspace{0.4cm}
	
\noindent We begin with a pair of examples of $q$-commuting of operators. In a similar fashion one can construct examples of $q$-intertwining pairs of operators.
 
\begin{eg}
	Let $q$ be any non zero complex number. Now let $T_1=\begin{pmatrix}
		1&0\\
		1/2&q
	\end{pmatrix}$ and $T_2=\begin{pmatrix}
	0&0\\
	1/4&0
\end{pmatrix}$ in $M_2(\mathbb{C})$. Then $$T_1T_2=\begin{pmatrix}
0&0\\
q/4&0
\end{pmatrix}$$ and $$T_2T_1=\begin{pmatrix}
0&0\\
1/4&0
\end{pmatrix} .$$ Hence clearly $T_1T_2=qT_2T_1$. 
In fact for any complex numbers $a,b$ and $d$ taking $T_1=\begin{pmatrix}
	a&0\\
	b&qa
\end{pmatrix}$ and $T_2=\begin{pmatrix}
	0&0\\
	d&0
\end{pmatrix}$ in $M_2(\mathbb{C})$ gives an example of a $q$- commuting pair of operators $(T_1,T_2)$.
\end{eg}
	  The following example was given in \cite{K.M.}.
	\begin{eg}[\cite{K.M.}]
		Let $H^2(\mathcal{E})=H^2(\mathbb{D})\otimes \mathcal{E}$ be the $\mathcal{E}$-valued Hardy Hilbert space (similarly,  $L^2_a(\mathbb{D},\mathcal{E})=L^2_a\otimes\mathcal{E}$ denote the $\mathcal{E}$-valued Bergman space). Let $q$ be any complex number of modulus 1. Let $T_2=M_z\otimes I_{\mathcal{E}}$ and $T_1=C_q\otimes \mathcal{E}$ on $H^2(\mathcal{E})$ (similarly on $L^2_a(\mathbb{D})$) be such that $M_z$ is a multiplication operator on Hardy space  $H^2(\mathbb{D})$ and $C_q$ on $H^2(\mathbb{D})$ is defined as \[C_q(f)(z)=f(qz) \text{ for } f\in H^2(\mathbb{D})\; (\text{similarly } f\in L^2_a(\mathbb{D})),\; z\in \mathbb{D}.\] Then it can be easily verified that $T_1T_2=qT_2T_1$.
	\end{eg}
	To see more examples of $q$-commuting operators, an interested reader can refer \cite{K.M.}.  
	The following dual version of Parrot's theorem (\cite{parrot1}, Theorem 1) on quotient norms  plays an important role in this paper.
	
	\begin{thm}[{\cite{Sebestyen}, Theorem 1}]\label{dualparrot}
		Let ${K}$ and ${K}'$ be Hilbert spaces, $H \subseteq K$ and $H' \subseteq K'$ be subspaces, and $X : H \rightarrow K'$ and $X': H' \rightarrow K$ be given bounded linear transformations. Then there exists operator $Y:K \rightarrow K'$ extending $X$ so that $Y^*$ extends $X'$ if and only if the following identity holds true:
		\[
		\langle Xh,h' \rangle = \langle h, X'h' \rangle \quad \text{for all }h \in H \text{ and } h' \in H'.
		\]
		Moreover $Y$ can be of norm max$\{\|X\|, \|X'\|\}$ possible at most.
	\end{thm}
	
	The following result from \cite{Douglas1} will be useful.
	
	\begin{lem} [\cite{Douglas1}, Theorem 1] \label{Dlemma}
		Suppose that $\mathcal{S}$, $\mathcal{H}$, $\mathcal{K}$ are Hilbert spaces and $A\in \mathcal{B}(\mathcal{S}, \mathcal{K})$ and $B\in \mathcal{B}(\mathcal{H}, \mathcal{K})$. Then there exists a contraction $Z\in \mathcal{B}(\mathcal{S},\mathcal{H})$ satisfying $A=BZ$ if and only if $AA^*\leq BB^*$.
	\end{lem}
	Now we recall a generalization of the above lemma.	
	
	\begin{thm}[{\cite{Douglas}, Theorem 1}]\label{Dthm1}
		Let $\mathcal{H}_0, \mathcal{H}_1,\mathcal{H}_2$ and $\mathcal{K}$ be Hilbert spaces, and for $0\leq i \leq 2$ let $A_i$ be an operator mapping $\mathcal{H}_i$ into $\mathcal{K}$. Then there exist operators $Z_1$ and $Z_2$ that map $\mathcal{H}_0$ into $\mathcal{H}_1$ and $\mathcal{H}_2$, respectively, and that satisfy the two conditions 
		\begin{enumerate}
			\item  $A_1Z_1+A_2Z_2=A_0,$
			\item  $Z_1^*Z_1+Z_2^*Z_2\leq I_{\mathcal{H}_0}$
		\end{enumerate} 
		if and only if 
		$ A_1A_1^*+A_2A_2^*\geq A_0A_0^*. $
	\end{thm}
We will make numerous applications of the following famous theorem due to Sz.-Nagy and Foias.
 	
	\begin{thm}[Sz.-Nagy \& Foias, \cite{nagy1968dilatation}]\label{douglasintertwining}
		Suppose that for $i=1,2$, $T_i$ is a contraction acting on a Hilbert space $\mathcal{H}_i$ and $V_i$ is the unique minimal co-isometric extension of $T_i$ acting on the Hilbert space $\mathcal{K}_i\supseteq \mathcal{H}_i$. Let $X$ be an operator that maps $\mathcal{H}_1$ into $\mathcal{H}_2$ and satisfies the equation $XT_1=T_2X$. Then there exists an operator $Y$ mapping $\mathcal{K}_1$ into $\mathcal{K}_2$ such that 
		\[
		YV_1 = V_2Y, \quad Y\mathcal{H}_1\subseteq \mathcal{H}_2, \quad Y|_{\mathcal{H}_1}=X, \quad \|Y\|=\|X\|.	\] 
	\end{thm}
	In \cite{Douglas}, Douglas, Muhly and Pearcy gave an alternate proof to the above theorem by an application of the commutant lifting theorem.
	
	Let $S_0, T_0$ be contractions acting on Hilbert spaces $\mathcal{H}_0$ and $\mathcal{K}_0$ respectively. Let $U$ and $V$ be the unique minimal co-isometric extensions of $S_0$ and $T_0$ acting on the Hilbert spaces $\mathcal{H}$ and $\mathcal{K}$ respectively. Let $P_n$ and $Q_n$ be the orthogonal projections of the spaces $\mathcal{H}$ and $\mathcal{K}$ to the corresponding subspaces 
	$ \bigvee\limits_{k=0}^n U^{*k}(\mathcal{H}_0)$ and $\bigvee\limits_{k=0}^n V^{*k}(\mathcal{K}_0)$ respectively, for $n = 0, 1, 2,\dots$ . The following theorem is a generalization of the commutant lifting theorem due to Sz. Nagy and Foias.
	  
		\begin{thm}[\cite{Sebestyen1}, Theorem 3]\label{sebintlift}
		Let $S_0$ on Hilbert space $\mathcal{H}_0$ and $T_0$ on Hilbert space $\mathcal{K}_0$ be any cotractions $R : \mathcal{H} \to \mathcal{H}$ be any contraction operator that commutes with all of the projections $\{P_n\}_{n=0}$, let $X_0 : \mathcal{H}\to\mathcal{K}$ be an operator that satisfies 
		\[T_0X_0=X_0S_0RP_0.\] 
		Then there exists an operator $X:\mathcal{H}\to \mathcal{K}$ which extends $X_0$ has the same norm as $X_0$ and intertwines $V$ and $UR$: 
\[ VX=XUR \]	
	\end{thm}

\vspace{0.1cm}	
	
\section{The $q$-commuting and $q$-intertwining co-isometric extensions}

\vspace{0.4cm}
	
\noindent In this Section, we assume the existence of a co-isometric extension of a non-zero contraction $T$ and find an extension of its $q$-commutant when $0< q \leq \dfrac{1}{\|T\|}$. For a contraction $T$ acting on a Hilbert space $\HS$, if $(V,\mathcal{K})$ is the minimal co-isometric extension of $T$ then $\mathcal{K}=\bigvee_{n=0}^{\infty}(V)^{*n}\mathcal{H}$. Suppose $K_n = \bigvee\limits_{m=0}^n V^{*m}\mathcal{H}$. We begin with the simpler case when $q$ belongs to the deleted closed disk $\overline{\mathbb D} \setminus \{0 \}$. 

\begin{thm}\label{qcommlift}
	Let $T_1$ be a contraction and $T_2$ be an operator on a Hilbert space $\mathcal{H}$. Suppose $T_1T_2=qT_2T_1$, where $0<|q|\leq 1$. Let $V$ on $\mathcal{K}$ be the minimal co-isometric extension of $T_1$. Then there exists a bounded linear operator $S:\mathcal{K} \rightarrow \mathcal{K}$ such that 
	\begin{itemize}
		\item[(1)] $VS = qSV$
		\item[(2)] $\|S\|=\|T_2\|$ and
		\item[(3)] $\mathcal{H}$ is invariant under $S$ and
		$S|_{\mathcal{H}}=T_2$.
	\end{itemize}
\end{thm}
\begin{proof}
	This result follows as a consequence of Theorem \ref{sebintlift}. Considering $qI_{\mathcal{K}}$ in place of $R$, $T_1$ in place of $T_0$ as well as $S_0$ and $T_2$ in place of $X_0$ in Theorem \ref{sebintlift}, there exists a bounded linear operator $S:\mathcal{K}\to \mathcal{K}$ such that the conditions $(1)$ and $(2)$ hold and $S|_{\mathcal{H}}=T_2$. As a consequence, for any $h\in \mathcal{H}$, $$V^nS^mh=V^nT_2^mh=T_1^nT_2^mh.$$ Hence $(3)$ holds and the proof is complete.   
\end{proof}
Removing the minimality constraint from Theorem \ref{qcommlift}, we obtain the following generalized version.
\begin{thm}\label{qcomexminno}
	Let $T_1$ be a contraction and $T_2$ be an operator on a Hilbert space $\mathcal{H}$. Suppose $T_1T_2=qT_2T_1$, where $0<|q|\leq 1$. Let $V$ on $\mathcal{K}$ be a co-isometric extension of $T_1$. Then there exists a bounded linear operator $S:\mathcal{K} \rightarrow \mathcal{K}$ such that 
	\begin{itemize}
		\item[(1)] $VS = qSV$
		\item[(2)] $\|S\|=\|T_2\|$ and
		\item[(3)] $\mathcal{H}$ is invariant under $S$ and
		$S|_{\mathcal{H}}=T_2$.
	\end{itemize}
\end{thm}
\begin{proof}
	The idea of proof is same as that of Corollary 4.1 in \cite{Douglas}.
	Let $\mathcal{K}'\subseteq \mathcal{K}$ be the smallest reducing subspace for $V$ that contains $\mathcal{H}$. If we put  $V'=V|_{\mathcal{K}'}$, then $(V',\mathcal{K}')$ is the minimal co-isometric extension of $T$. Thus by Theorem \ref{qcommlift}, there exists $S'\in \mathcal{B}(\mathcal{K}')$ such that $\|S'\|=\|T_2\|$, $V'S'=qS'V'$, $\mathcal{H}$ is invariant under $S'$ and $S'|_{\mathcal{H}}=T_2$. Now we define $S:\mathcal{K}\to \mathcal{K}$ as $S(k)=S'(k)$ for any $k\in \mathcal{K}'$ and $S(k)=0$ on $\mathcal{K}\ominus\mathcal{K}'$. Hence clearly $S$ satisfies properties $(1),(2)$ and $(3)$ as desired. This completes the proof.
\end{proof}
	In Theorem \ref{qcommlift}, in particular if $T_2$ is also a contraction, then one might expect $S$ to be a co-isometric extension of $T_2$. But for $|q|\neq 1$, this is not possible as $\|VS\|=1$ and $\|qSV\|=|q|$. Although, we can expect the existence of co-isometric extensions $V_1,V_2,qV_q$ of $T_1,T_2,qT_1$ respectively, such that $V_1V_2=qV_2V_q$. In this Section, we obtain such $V_1,V_2,qV_q$ when $T_1,T_2$ are $q$-commuting contractions with $\|T_2\|<1$ and $0<|q|\leq \dfrac{1}{\|T_1\|}$, see Theorem \ref{qcoisoext}. Indeed, this is one of the main theorems of this paper, but before going to that we state and prove another result which is also one of our main results.
	 
	\begin{thm}\label{Dmain}
		Let $T$ be a non-zero contraction on Hilbert space $\mathcal{H}$ and $X\in \mathcal{B}(\mathcal{H})$ be such that $TX=qXT$ for $\, 0<|q|\leq \dfrac{1}{\|T\|}$. Suppose $(V,\mathcal{K}')$ and $(qV_q,\mathcal{K})$ are the minimal co-isometric extensions of $T$ and $qT$ respectively. Then there exists an operator $Y\in \mathcal{B}(\mathcal{K},\mathcal{K}')$ such that 
		\[ VY=qYV_q,\; Y(\mathcal{H})\subseteq \mathcal{H},\; Y|_{\mathcal{H}}=X \text{ and } \|X \|=\|Y\|  .\] 
	\end{thm}
	\begin{proof}
		We follow the idea of the proof of Theorem 2 in \cite{Sebestyen}. Since $(V,\mathcal{K}')$ and $(qV_q, \mathcal{K})$ are the minimal co-isometric extensions of $T$ and $qT$ respectively, we have 
		\[
		\mathcal{K}'= \bigvee_{n=0}^{\infty}V^{*n}\mathcal{H} \;\text{  and  }\; \mathcal{K}=\bigvee_{n=0}^{\infty}(qV_q)^{*n}\mathcal{H}.
		\]
		Let $K_n = \bigvee\limits_{\ell = 0}^{n}(qV_q)^{*\ell}\mathcal{H}$ and $K_n' = \bigvee\limits_{m=0}^n V^{*m}\mathcal{H}$. Then $\bigcup\limits_{n = 0}^{\infty} K_n$ is dense in $\mathcal{K}$ and $\bigcup\limits_{n = 0}^{\infty}K_n'$ is dense in $\mathcal{K}'$. Consider the orthogonal projections $P_n : \mathcal{K}\rightarrow K_n$ and $P_n': \mathcal{K}' \rightarrow K_n'$ for all $n \geq 0$. Clearly 
		\begin{equation}\label{eq1'}
		P_{n+1}(qV_q)^* = (qV_q)^*P_{n} \;\text{ and }\; P_{n+1}'V^* = V^*P_n'.
		\end{equation}
		We construct the required operator $Y:\mathcal{K} \rightarrow \mathcal{K}'$ inductively by finding operators $Y_n : K_n \rightarrow K_n'$ for all $n \in \mathbb N$.\\
		
		\noindent\textbf{Claim.} For any $h_1,h_2\in \mathcal{H}$, $\langle (qV_q)^*X^*V(V^*h_1), h_2 \rangle = \langle (V^*h_1), Xh_2 \rangle$.\\
		
		\noindent \textbf{Proof of Claim.} Since $V$ and $qV_q$ are the minimal co-isometric extensions of $T$ and $qT$ respectively, then for $h_1,h_2\in \mathcal{H}$ we have 
		\begingroup
		\allowdisplaybreaks
		\begin{align*}
		\langle (qV_q)^*X^*V(V^*h_1), h_2 \rangle  = \langle (qV_q)^*X^*h_1,h_2 \rangle = \langle X^*h_1, (qV_q)h_2 \rangle & = \langle X^*h_1, qTh_2 \rangle 		\\& = \langle h_1, qXTh_2 \rangle\\
		& = \langle h_1, TXh_2 \rangle \; \; (\text{since } TX=qXT)\\
		& = \langle h_1, VXh_2 \rangle\\
		& = \langle V^*h_1, Xh_2 \rangle.
		\end{align*}
		\endgroup
		This completes the proof of the claim.\\
		
		\noindent Therefore, by Theorem \ref{dualparrot}, there exists an operator $Y_1^*\in \mathcal{B}(K_1',K_1)$ such that $Y_1|_{\mathcal{H}}=X$, $Y_1^*|_{V^*\mathcal{H}}=(qV_q)^*X^*V|_{V^*\mathcal{H}}$ and $\|Y_1\|\leq \text{max}\{\|X\|,\|(qV_q)^*X^*V|_{V^*\mathcal{H}}\| \}=\|X\|$. Since $Y_1$ is an extension of $X$, we have $\|Y_1\|=\|X\|$. Also, we have 
		\begingroup
		\allowdisplaybreaks
		\begin{align}\label{eq2'}
		Y_1^*P'_1V^* &= Y_1^*V^*P'_0 \hspace{7.05cm} (\text{from }\eqref{eq1'})\notag\\
		& = (qV_q)^*X^*VV^*P'_0 \hspace{1.8cm}  \left(\text{since }Y_1^*|_{V^*\mathcal{H}}=(qV_q)^*X^*V|_{V^*\mathcal{H}}\right)\notag\\
		& = (qV_q)^*X^*P'_0.
		\end{align}
		\endgroup
		
		\noindent Suppose there is $Y_{n-1}$ such that $Y_{n-1}^*: K_{n-1}'\rightarrow K_{n-1}$ satisfies
		\begin{align}
			&(a)\;  Y_{n-1}|_{K_{n-2}}=Y_{n-2}, \notag \\
			& (b) \; Y^*_{n-1}|_{V^*K_{n-2}'}=(qV_q)^*Y_{n-2}^*V|_{V^*K_{n-2}'}, \notag
			\\& (c) \; Y_{n-1}^*P_{n-1}'V^* = (qV_q)^*Y_{n-2}^*P_{n-2}', \label{eqn:011} \\
			& (d) \; \|Y_{n-1}\|=\|Y_{n-2}\| \notag.
		\end{align}
	
	\noindent\textbf{Claim.}	For any $h_1\in K_{n-1}' \text{ and }h_2 \in K_{n-1}$, we have $\langle (qV_q)^*Y_{n-1}^*V(V^*h_1), h_2 \rangle = \langle V^*h_1, Y_{n-1}h_2 \rangle.$\\
	
	\noindent \textbf{Proof of Claim.} Suppose $h_1\in K_{n-1}' \text{ and }h_2 \in K_{n-1}$. Then
		\begingroup
		\allowdisplaybreaks
		\begin{align*}
		\langle (qV_q)^*Y_{n-1}^*V(V^*h_1), h_2 \rangle & = \langle (qV_q)^*Y_{n-1}^*h_1,h_2 \rangle \\
		& = \langle P_{n-1}(qV_q)^*Y_{n-1}^*h_1, h_2 \rangle\; (\text{since }(qV_q)^*K_{n-1}\subseteq K_{n}) \\
		& = \langle (qV_q)^*P_{n-2}Y_{n-1}^*h_1, h_2 \rangle \;(\text{using }\eqref{eq1'}) \\
		& = \langle h_1, Y_{n-1}P_{n-2}(qV_q)h_2 \rangle\\
		& = \langle P_{n-2}'h_1, Y_{n-1}P_{n-2}(qV_q)h_2 \rangle\;(\text{since }Y_{n-1}K_{n-2}\subseteq K_{n-2}')\\
		& = \langle P_{n-2}'h_1, Y_{n-2}P_{n-2}(qV_q)h_2 \rangle \; (\text{using (a)})\\
		& = \langle Y_{n-2}^*P_{n-2}'h_1, P_{n-2}(qV_q)h_2 \rangle \\
		& = \langle Y_{n-2}^*P_{n-2}'h_1, (qV_q)h_2 \rangle\; (\text{since }Y_{n-2}^*K_{n-2}'\subseteq K_{n-2}) \\
		& = \langle (qV_q)^*Y_{n-2}^*P_{n-2}'h_1, h_2 \rangle\\
		& = \langle Y_{n-1}^*P_{n-1}'V^*h_1, h_2 \rangle \; (\text{using (c)})\\
		& = \langle V^*h_1, Y_{n-1}h_2 \rangle.
		\end{align*}
		\endgroup
		The last inequality follows from the fact that $Y_{n-1}h_2\in K_{n-1}'$, for any $h_2\in K_{n-1}$. This completes the proof of claim.\\
		
 Hence, Theorem \ref{dualparrot} guarantees the existence of $Y_n: K_n \to K_n'$ satisfying 
		\begin{itemize}
			\item[(i)] $Y_{n}|_{K_{n-1}}=Y_{n-1}$,
			\item[(ii)] $Y^*_{n}|_{V^*K_{n-1}'}=(qV_q)^*Y_{n-1}^*V|_{V^*K_{n-1}'}$,
			\item[(iii)] $Y_{n}^*P_{n}'V^* = (qV_q)^*Y_{n-1}^*P_{n-1}'$,
			\item[(iv)] $\|Y_{n}\|=\|Y_{n-1}\|=\|X\|$.
		\end{itemize} 
Note that condition-$(iii)$ above is obtained in the following way:
	\[
		Y_n^*P_n'V^*  = Y_{n}^*V^*P_{n-1}' = (qV_q)^*Y_{n-1}^*VV^*P_{n-1}' = (qV_q)^*Y_{n-1}^*P_{n-1}',
		\]
		where the second last equality follows from condition-$(c)$ of (\ref{eqn:011}). Thus, by induction the result holds for any natural number $n$. Now define $Y_0:\bigcup\limits_{n=0}^{\infty} \mathcal{K}_n \rightarrow \bigcup\limits_{n=0}^{\infty} \mathcal{K}_n'$ such that $Y_0|_{K_n}=Y_n$. This is well defined as $Y_{n}|_{K_{n-1}} = Y_{n-1}$. Since $\bigcup\limits_{n = 0}^{\infty}K_n'$ is dense in $\mathcal{K}'$, by continuity $Y_0$ extends to an operator $Y:\mathcal{K}\to \mathcal{K}'$.  From (iii) we have $VY_nP_n=Y_{n-1}P_{n-1}qV_q$ and one can easily check that $Y_nP_n$ converges to $Y$ in the strong operator topology. Therefore, $
		VY = qYV_q.
		$
		Clearly $\|Y\|=\|X\|$. This completes the proof of the theorem.
	\end{proof}

 If we drop the minimality conditions on co-isometric extension in Theorem \ref{Dmain}, then we have the following result.
	\begin{cor}\label{q-exten}
		Let $T$ be a contraction on Hilbert space $\mathcal{H}$ and $X\in \mathcal{B}(\mathcal{H})$ be such that $TX=qXT$ for $0<|q|\leq \dfrac{1}{\|T\|}$. Suppose $(V,\mathcal{K}')$ and $(qV_q,\mathcal{K})$ are co-isometric extensions of $T$ and $qT$ respectively. Then there exists an operator $Y\in \mathcal{B}(\mathcal{K},\mathcal{K}')$ such that 
		\[ VY=qYV_q,\; Y(\mathcal{H})\subseteq \mathcal{H},\; Y|_{\mathcal{H}}=X \text{ and } \|X \|=\|Y\|  .\] 
	\end{cor}
	\begin{proof}
		Let $\mathcal{K}_1\subseteq \mathcal{K}$ and $\mathcal{K}_1'\subseteq \mathcal{K}'$ be the smallest reducing subspaces containing $\mathcal{H}$, for $qV_q$ and $V$ respectively. Let $V_1=V|_{\mathcal{K}_1}$ and $qV_2=(qV_q)|_{\mathcal{K}_1'}$. Then $V_1$ and $qV_2$ are minimal co-isometric extensions of $T$ and $qT$ respectively. By Theorem \ref{Dmain}, there exists $Y_1\in \mathcal{B}(\mathcal{K}_1, \mathcal{K}_1')$ such that $V_1Y_1 = qY_1V_2$, $Y_1\mathcal{H}\subseteq \mathcal{H}$, $Y_1|_{\mathcal{H}}= X$ and $\|Y_1\|=\|X\|$. 
		Define $Y:\mathcal{K} \to \mathcal{K}'$ as $Y(k)=Y_1(k)$ for $k \in \mathcal{K}_1$ and $Y(k)=0$ for $k \in \mathcal{K} \ominus \mathcal{K}_1$. Clearly $Y(\mathcal{H})\subseteq \mathcal{H},\; Y|_{\mathcal{H}}=X \text{ and } \|X \|=\|Y\|$. 
		As $qV_q$ reduces $\mathcal{K}_1$, for $k \in \mathcal{K} \ominus \mathcal{K}_1$, $qV_q(k)\in \mathcal{K}\ominus \mathcal{K}_1$. Hence $VY(k)=0, \; qYV_q(k)=0$ and the proof is complete.
	\end{proof}

Now we are in a position to present one of our main results.
\begin{thm}\label{qcoisoext}
	Let $T_1$ be any contraction and $T_2$ be a strict contraction on Hilbert space $\mathcal{H}$, such that $T_1T_2=qT_2T_1$  where $0<|q|\leq \dfrac{1}{\|T_1\|}$. Then there exist a Hilbert space $\mathcal{K}$ containing $\mathcal{H}$ and co-isometric extensions $X_1, X_2, qX_q$ on $\mathcal{K}$ of $T_1,T_2,qT_1$ such that  $X_1X_2=qX_2X_q$.
	
	Moreover, if $T_1$ is also a strict contraction then $X_1,X_2,qX_q$ are pure co-isometries. 
\end{thm}
\begin{proof}
	Let $(V,\widetilde{\mathcal{K}})$ and $(qV_q,\widetilde{\mathcal{K}})$ be co-isometric extensions of $T_1$ and $qT_1 $ respectively (indeed, we can take $\widetilde{\mathcal{K}}=\ell^2(\mathcal{H})$). By Corollary \ref{q-exten}, there exists an operator $X$ on $\widetilde{\mathcal{K}}$ such that \begin{enumerate}
		\item $X|_{\mathcal{H}}=T_2$;
		\item $\|X\|=\|T_2\|$;
		\item $VX=qXV_q$.
	\end{enumerate} 
	Since $X$ is a strict contraction so $D_{X^*}$ is invertible on $\widetilde{\mathcal{K}}$. Consider the Hilbert space $\mathcal{K}=\bigoplus\limits_{0}^{\infty}\widetilde{\mathcal{K}}$. We embed $\mathcal{H}$ in $\mathcal{K}$ via a map $h \mapsto (h,0,0,\ldots)$.  Let $X_2, X_1$ and $qX_1$ be operators on $\mathcal{K}$ defined as follows 
	\begingroup
	\allowdisplaybreaks
	\begin{gather*}
	X_2=\begin{pmatrix}
	X&D_{X^*}&0&0&\cdots\\
	0&0&I_{\widetilde{\mathcal{K}}}&0&\cdots\\
	0&0&0&I_{\widetilde{\mathcal{K}}}&\cdots\\
	\vdots&\vdots&\vdots&\vdots&\ddots
	\end{pmatrix},\\
	X_1=\begin{pmatrix}
	V&0&0&\cdots \\
	0&D_{X^*}^{-1}VD_{X^*}&0&\cdots \\
	0&0&D_{X^*}^{-1}VD_{X^*}&\cdots \\
	\vdots&\vdots&\vdots&\ddots
	\end{pmatrix},\\
	qX_q=\begin{pmatrix}
	qV_q&0&0&\cdots \\
	0&D_{X^*}^{-1}VD_{X^*}&0&\cdots \\
	0&0&D_{X^*}^{-1}VD_{X^*}&\cdots \\
	\vdots&\vdots&\vdots&\ddots
	\end{pmatrix}.
	\end{gather*}
	\endgroup
	Clearly, $X_2,X_1$ and $qX_q$ are bounded linear operators on $\mathcal{K}$ and by construction $X_2$ is a pure co-isometry. Now we prove that $X_1$ and $qX_q$ are co-isometries. Since $V$ and $qV_q$ are co-isometries, it suffices to show that 
	\begin{equation}\label{T1}
	(D_{X^*}^{-1}VD_{X^*})(D_{X^*}^{-1}VD_{X^*})^*=I_{\widetilde{\mathcal{K}}}.
	\end{equation}
	Using $qXV_q=VX$, $VV^*=I_{\widetilde{\mathcal{K}}}$ and $(qV_q)(qV_q)^*=I_{\widetilde{\mathcal{K}}}$ we have
	\begingroup
	\allowdisplaybreaks
	\begin{align*}
	(D_{X^*}^{-1}VD_{X^*})(D_{X^*}^{-1}VD_{X^*})^*&=D_{X^*}^{-1}VD_{X^*}^2V^*D_{X^*}^{-1}\\
	&=D_{X^*}^{-1}V(I_{\widetilde{\mathcal{K}}}-XX^*)V^*D_{X^*}^{-1}\\
	&=D_{X^*}^{-1}(I_{\widetilde{\mathcal{K}}}-VX(VX)^*)D_{X^*}^{-1}\\
	&=D_{X^*}^{-1}(I_{\widetilde{\mathcal{K}}}-qXV_q(qXV_q)^*)D_{X^*}^{-1}\\
	&=D_{X^*}^{-1}(I_{\widetilde{\mathcal{K}}}-XX^*)D_{X^*}^{-1}\\
	&=I_{\widetilde{\mathcal{K}}}.
	\end{align*}
	\endgroup
	By direct computation and using $VX=qXV_q$, we get $X_1X_2=qX_2X_q$. We have that $V,qV_q$ and $X$ are extensions of $T_1,qT_1$ and $T_2$ respectively. Thus $X_1,qX_q$ and $X_2$ are extensions of $T_1,qT_1$ and $T_2$ respectively. 
	
	Again if $T_1$ is a strict contraction then we can assume $V$ to be a pure co-isometric extension of $T_1$ and we have 
	\begin{equation}\label{T2}
	(D_{X^*}^{-1}VD_{X^*})^n=D_{X^*}^{-1}V^nD_{X^*}.
	\end{equation}
	Since $V$ and $qV_q$ are pure, we obtain from equations (\ref{T1}) and (\ref{T2}) that $X_1$ and $qX_q$ are pure co-isometric extensions of $V$ and $qV_q$ respectively. The proof is now complete. 
\end{proof}

By an application of Theorem \ref{qcoisoext}, we obtain the following two lemmas.

\begin{lem}
	Let $T_1,T_2\in \mathcal{B}(\mathcal{H})$ be contractions such that $T_1 \neq 0$, $\|T_2\|<1 $ and $T_1T_2=qT_2T_1$ for some $0<|q|\leq \dfrac{1}{\|T_1\|}$. Suppose $\widetilde{X}_2$ is a co-isometric extension of $T_2$ on $\widetilde{\mathcal{K}}$. Then there exist co-isometric extensions $V_1,V_2$ and $qV_q$ of $T_1,\widetilde{X}_2$ and $qT_1$ on a Hilbert space $\mathcal{K}\supseteq\widetilde{\mathcal{K}}$ respectively such that $V_1V_2 = qV_2V_q$. 
\end{lem}
\begin{proof}
	Let $\widetilde{X}_2=X_0\oplus Y_2$ on $\widetilde{\mathcal{K}}=\widetilde{\mathcal{K}}_0\oplus \mathcal{R}_2$, where $X_0$ is the minimal co-isometric extension of $T_2$. By Theorem \ref{qcoisoext}, there exist co-isometric extensions $Z_1,Z_2,qZ_q$ of $T_1,T_2, qT_1$ respectively on some Hilbert space $\mathfrak{L}$ such that $Z_1Z_2 = qZ_2Z_q$. Let $\mathcal{K}= \mathfrak{L}\oplus \mathcal{R}_2$, $X_2=Z_2 \oplus Y_2, X_1= Z_1 \oplus I$ and $qX_q = qZ_q \oplus I$. Clearly, $X_1X_2 = qX_2X_q$, $X_1,X_2$, $qX_q $ are co-isometries and $X_1|_{\mathcal{H}}=T_1, qX_q|_{\mathcal{H}}=qT_1$. Since minimal co-isometric extension of a contraction is unique up to isomorphism and every co-isometric extension of a contraction contains a minimal co-isometric extension, $Z_2$ is a co-isometric extension of $X_0$. Therefore, $X_2$ is a co-isometric extension of $\widetilde{X}_2$. This completes the proof.
\end{proof}

\begin{lem}
	Let $T_1,T_2\in \mathcal{B}(\mathcal{H})$ be contractions such that $T_1 \neq 0$, $\|T_2\|<1 $ and $T_1T_2=qT_2T_1$ for some $0<|q|\leq \dfrac{1}{\|T_1\|}$. Suppose $\widetilde{X}_1$ on $\widetilde{\mathcal{K}}_1$ and $\widetilde{qX}_q$ on $\widetilde{\mathcal{K}}_2$ are co-isometric extensions of $T_1$ and $qT_1$ respectively. Then there exist co-isometric extensions $X_1,X_2$ and $qX_q$ of $\widetilde{X}_1,T_2$ and $\widetilde{qX}_q$ on a Hilbert space $\mathcal{K}\supseteq\widetilde{\mathcal{K}}$ respectively such that $X_1X_2 = qX_2X_q$.
\end{lem}
\begin{proof}
	Suppose $\widetilde{X}_1=X_{10}\oplus Y_1$ on $\widetilde{\mathcal{K}}_1=K_{10}\oplus \mathcal{R}_1$ and $\widetilde{qX}_q =X_{q0}\oplus Y_2$ on $\widetilde{\mathcal{K}}_2=K_{20}\oplus \mathcal{R}_2$, where $X_{10}$ and $X_{q0}$ are the minimal co-isometric extensions of $T_1$ and $qT_1$ respectively. By Theorem \ref{qcoisoext}, there exist co-isometric extensions $Z_1,Z_2,qZ_q$ of $T_1,T_2, qT_1$ respectively on $\mathfrak{L}$ such that $Z_1Z_2 = qZ_2Z_q$. Let $\mathcal{K}= \mathfrak{L}\oplus \mathcal{R}_1\oplus \mathcal{R}_2$, $X_2=Z_2 \oplus I\oplus I, X_1= Z_1 \oplus Y_1\oplus Y_2$ and $qX_q = qZ_q \oplus Y_1 \oplus Y_2$. Clearly, $X_1X_2 = qX_2X_q$, $X_1,X_2$, $qX_q $ are co-isometries and $X_2|_{\mathcal{H}}=T_2$. Further $X_1$ and $qX_q$ are co-isometric extensions of $\widetilde{X}_1$ and $\widetilde{qX}_q$ respectively. This is because, $X_{10}$, $X_{q0}$ are the minimal co-isometric extensions of $T_1$, $qT_1$ respectively and also $\widetilde{X}_1,\widetilde{qX}_q$ are co-isometric extenstions of $T_1$, $qT_1$ respectively.
\end{proof}

We conclude this Section with a similar co-extension result for $q$-intertwining operators when $q$ is a complex number of modulus one. 

\begin{thm}\label{coisoextn}
	Suppose $A\in \mathcal{B}(\mathcal{H}_1,\mathcal{H}_2)$ is a strict contraction and $T_i\in \mathcal{B}(\mathcal{H}_i)$ is a contraction for $i = 1,2$ such that $AT_1=qT_2A$, where $|q|=1$. Then there exist co-isometric extensions $Y\in \mathcal{B}(\mathcal{K}_1, \mathcal{K}_2)$ of $A$ and $X_i\in\mathcal{B}(\mathcal{K}_i)$ of $T_i$ for $i=1,2$ such that $YX_1=qX_2Y$. Moreover, $Y$ is a pure co-isometry and if $T_i$ is a strict contraction then $X_i$ is pure co-isometry.
\end{thm}
\begin{proof}
	Consider 
	\[
	\widetilde{T}=\begin{pmatrix}
		T_1 & 0\\
		0 & T_2
	\end{pmatrix},\quad \widetilde{A}=\begin{pmatrix}
		0 & 0\\
		A & 0
	\end{pmatrix} \text{ on }\mathcal{H}_1\oplus\mathcal{H}_2.
	\]
	Suppose $\widetilde{V}=\begin{pmatrix}
		V_1 & 0\\
		0 & V_2
	\end{pmatrix}$, where $V_i$ on $\ell^2(\mathcal{H}_i)(=\mathfrak{L}_i)$ is a co-isometric extension of $T_i$ for $i=1,2$. 
	Clearly $\widetilde{A}\widetilde{T}= q\widetilde{T}\widetilde{A}$ and $\widetilde{V}$ is a co-isometric extension of $\widetilde{T}$ on $K=\mathfrak{L}_1\oplus \mathfrak{L}_2$. By Corollary \ref{qcomexminno}, there exists an operator $\widetilde{Y}$ in $\mathcal{B}(K)$ such that 
	\[
	\widetilde{Y}\widetilde{V}=q\widetilde{V}\widetilde{Y}, \quad \widetilde{Y}|_{\mathcal{H}_1\oplus\mathcal{H}_2}= \widetilde{A}, \quad \|\widetilde{Y}\|=\|\widetilde{A}\|=\|A\|.
	\]
	The block matrix of $\widetilde{Y}$ with respect to $\mathfrak{L}_1 \oplus \mathfrak{L}_2$ is $\begin{pmatrix}
		Z_1 & Z_2\\
		B & Z_3
	\end{pmatrix}$. Hence we have $BV_1=qV_2B$. Now clearly $\|B\|\leq \|\widetilde{Y}\|=\|A\|$. Further $\widetilde{Y}|_{\mathcal{H}_1\oplus\mathcal{H}_2}= \widetilde{A}$ implies that $B|_{\mathcal{H}_1}=A$. Therefore, $\|A\|\leq \|B\|$ and hence $\|B\|=\|A\|<1$. Since $\|B\|$ is less than $1$, $D_{B^*}$ is invertible. Consider the following operators  
	\begingroup
	\allowdisplaybreaks
	\begin{align*}
		& Y = \begin{pmatrix}
			B & \dfrac{1}{q}D_{B^*}& 0 & 0 & 0 & \cdots\\
			0 & 0 & I & 0 & 0 & \cdots\\
			0 & 0 & 0 & I & 0 & \cdots\\
			0 & 0 & 0 & 0 & I & \cdots\\
			\vdots & \vdots & \vdots & \vdots & \vdots & \ddots
		\end{pmatrix} : \mathfrak{L}_1 \oplus \ell^2(\mathfrak{L}_2) \to \mathfrak{L}_2 \oplus \ell^2(\mathfrak{L}_2),\\
		& X_1 = \begin{pmatrix}
			V_1 & 0 & 0 & 0 &\cdots\\
			0 & qD_{B^*}^{-1}V_2D_{B^*} & 0 & 0 & \cdots\\
			0 & 0 & D_{B^*}^{-1}V_2D_{B^*} & 0 & \cdots\\
			0 & 0 & 0 & D_{B^*}^{-1}V_2D_{B^*} &  \cdots\\
			\vdots & \vdots & \vdots & \vdots & \ddots
		\end{pmatrix} \text{ on } \mathfrak{L}_1 \oplus \ell^2(\mathfrak{L}_2), \\
		& X_2 = \begin{pmatrix}
			V_2 & 0 & 0 & 0 & \cdots\\
			0 & \dfrac{1}{q}D_{B^*}^{-1}V_2D_{B^*} & 0 & 0 & \cdots\\
			0 & 0 & \dfrac{1}{q}D_{B^*}^{-1}V_2D_{B^*} & 0 &  \cdots\\
			0 & 0 & 0 & \dfrac{1}{q}D_{B^*}^{-1}V_2D_{B^*} &  \cdots\\
			\vdots & \vdots & \vdots & \vdots & \ddots
		\end{pmatrix} \text{ on } \mathfrak{L}_2 \oplus \ell^2(\mathfrak{L}_2). 
	\end{align*} 
	\endgroup
	It can be easily verified that $X_1,X_2$ and $Y$ satisfy the required conditions. The proof is complete.
	
\end{proof}

\section{$q$-commutant lifting}

\vspace{0.4cm}

\noindent In the previous Section, we assumed the existence of co-isometric extension of a contraction and established a few $q$-commutant lifting results. Here we study the existence of a $q$-commutant of the minimal isometric lift of a contraction. We begin with a few preparatory results.     
\begin{prop}\label{Dgenprop}
	Let $T_1\in \mathcal{B}(\mathcal{H}_1,\mathcal{H}_1')$ and $T_2\in \mathcal{B}(\mathcal{H}_2,\mathcal{H}_2')$ be contractions. Suppose $X$ is an operator from $\mathcal{H}_1$ to $\mathcal{H}_2'$. Then \[
	Y = \begin{pmatrix}
	T_1 & 0\\
	X & T_2
	\end{pmatrix}:\mathcal{H}_1\oplus \mathcal{H}_2 \to \mathcal{H}_1'\oplus \mathcal{H}_2'\]
	is a contraction if and only if $X = D_{T_2^*}CD_{T_1}$ for some contraction $C : \mathcal{H}_1 \to \mathcal{H}_2'$.
\end{prop}
\begin{proof}

This is a variant of Proposition 2.2 in \cite{Douglas} and we implement a similar idea to prove it.
	The operator $Y = \begin{pmatrix}
	T_1 & 0\\
	X & T_2
	\end{pmatrix}$ is a contraction if and only if 
	$
	Y^*Y \leq \begin{pmatrix}
	I_{\mathcal{H}_1} & 0\\
	0 & I_{\mathcal{H}_2}
	\end{pmatrix},
	$
	that is, if and only if 
	\begingroup
	\allowdisplaybreaks
	\begin{align}\label{eqmatrix}
	&\left(
	\begin{pmatrix}
	T_1 & 0\\
	0 & 0
	\end{pmatrix} + \begin{pmatrix}
	0 & 0\\
	X & T_2
	\end{pmatrix}
	\right)^*\left( \begin{pmatrix}
	T_1 & 0\\
	0 & 0
	\end{pmatrix} + \begin{pmatrix}
	0 & 0\\
	X & T_2
	\end{pmatrix} \right) \leq \begin{pmatrix}
	I_{\mathcal{H}_1} & 0\\
	0 & I_{\mathcal{H}_2}
	\end{pmatrix}\notag\\
	\iff & \begin{pmatrix}
	T_1^* & 0\\
	0 & 0
	\end{pmatrix}\begin{pmatrix}
	T_1 & 0\\
	0 & 0
	\end{pmatrix} + \begin{pmatrix}
	0 & 0\\
	X & T_2
	\end{pmatrix}^*\begin{pmatrix}
	0 & 0\\
	X & T_2
	\end{pmatrix} \leq \begin{pmatrix}
	I_{\mathcal{H}_1} & 0\\
	0 & I_{\mathcal{H}_2}
	\end{pmatrix}\notag\\
	\iff & \begin{pmatrix}
	T_1^*T_1 & 0\\
	0 & 0
	\end{pmatrix}+ \begin{pmatrix}
	0 & 0\\
	X & T_2
	\end{pmatrix}^*\begin{pmatrix}
	0 & 0\\
	X & T_2
	\end{pmatrix} \leq \begin{pmatrix}
	I_{\mathcal{H}_1} & 0\\
	0 & I_{\mathcal{H}_2}
	\end{pmatrix}\notag\\
	\iff & \begin{pmatrix}
	0 & 0\\
	X & T_2
	\end{pmatrix}^*\begin{pmatrix}
	0 & 0\\
	X & T_2
	\end{pmatrix} \leq \begin{pmatrix}
	I_{\mathcal{H}_1}-T_1^*T_1 & 0\\
	0 & I_{\mathcal{H}_2}
	\end{pmatrix}.
	\end{align}
	\endgroup
	By Lemma \ref{Dlemma}, Equation \eqref{eqmatrix} holds if and only if there exists a contraction \[Z = \begin{pmatrix}
	Z_{11} & Z_{12}\\
	Z_{21} & Z_{22}
	\end{pmatrix}: \mathcal{H}_1 \oplus \mathcal{H}_2 \to  \mathcal{H}_1' \oplus \mathcal{H}_2'
	\]
	such that 
		\begin{equation}\label{eqmatrix1}
	\begin{pmatrix}
	0 & 0\\
	X & T_2
	\end{pmatrix}^* = \begin{pmatrix}
	D_{T_1} & 0\\
	0 & I_{\mathcal{H}_2}
	\end{pmatrix}\begin{pmatrix}
	Z_{11} & Z_{12}\\
	Z_{21} & Z_{22}
	\end{pmatrix}^* = \begin{pmatrix}
	D_{T_1}Z_{11}^* & D_{T_1}Z_{21}^*\\
	Z_{12}^* & Z_{22}^*
	\end{pmatrix}.
	\end{equation}
	Equation \eqref{eqmatrix1} holds if and only if there exists a contraction $Z$ such that $Z_{12} =0$, $Z_{22}=T_2$, $X=Z_{21}D_{T_1}$ and $D_{T_1}Z_{11}^*=0$. This is equivalent to the existence of a contraction $Z' =\begin{pmatrix}
	0 & 0\\
	Z_{21} & T_2
	\end{pmatrix}$ satisfying $X = Z_{21}D_{T_1}$. Now $Z'$ is a contraction if and only if $Z'Z'^*\leq I_{\mathcal{H}_1'\oplus \mathcal{H}_2'}$ if and only if $Z_{21}Z_{21}^*+T_2T_2^*\leq I_{\mathcal{H}_2'}$, that is, $Z_{21}Z_{21}^*\leq I_{\mathcal{H}_2'}-T_2T_2^*$. Thus, we have that $Y = \begin{pmatrix}
	T_1 & 0\\
	X & T_2
	\end{pmatrix}$ is a contraction if and only if there exists \[Z' =\begin{pmatrix}
	0 & 0\\
	Z_{21} & T_2
	\end{pmatrix}:\mathcal{H}_1\oplus \mathcal{H}_2\to \mathcal{H}_1'\oplus \mathcal{H}_2'\] satisfying $Z_{21}Z_{21}^*\leq I_{\mathcal{H}_2'}-T_2T_2^*$ and $X = Z_{21}D_{T_1}$. 
	By Lemma \ref{Dlemma}, $Z_{21}Z_{21}^*\leq I_{\mathcal{H}_2'}-T_2T_2^*$ holds if and only if there exists a contraction $C:\mathcal{H}_1\to \mathcal{H}_2'$ such that $Z_{21}=D_{T_2^*}C$. Therefore, $Y$ is a contraction if and only if $X=D_{T_2^*}CD_{T_1}$ for some contraction $C$ and the proof is complete. 	
\end{proof}

Following the same technique as in the proof of Theorem 3 in \cite{Douglas}, we obtain a generalized version of the above result for intertwining operators.

\begin{thm}\label{qpart}
	Let $T\in\mathcal{B}(\mathcal{H}_1)$, $T'\in\mathcal{B}(\mathcal{H}_1')$ and $T_2\in\mathcal{B}(\mathcal{H}_1,\mathcal{H}_1')$ be contractions such that $T_2T=T'T_2$. Suppose $V=\begin{pmatrix}
	T&0\\
	S&0
	\end{pmatrix}$ on $\mathcal{H}_1\oplus \mathcal{H}_2$, $V'=\begin{pmatrix}
	T'&0\\
	S'&0
	\end{pmatrix}$ on $\mathcal{H}_1'\oplus \mathcal{H}_2'$ such that $T^*T+S^*S=I_{\mathcal{H}_1}$ and $T'^*T'+S'^*S'=I_{\mathcal{H}_1'}$. Then there exists $Y=\begin{pmatrix}
	T_2&0\\
	A&B
	\end{pmatrix}$ from $\mathcal{H}_1\oplus \mathcal{H}_2$ to $\mathcal{H}_1'\oplus \mathcal{H}_2'$ such that $YV=V'Y$ and $\|Y\|=\|T_2\|$.  
\end{thm}
\begin{proof}
	Without loss of generality assume $\|T_2\|=1$. We want to find a matrix $Y=\begin{pmatrix}
	T_2&0\\A&B
	\end{pmatrix}$ from $\mathcal{H}_1\oplus \mathcal{H}_2$ to $\mathcal{H}_1'\oplus \mathcal{H}_2'$ such that $YV=V'Y$, that is, to find $A$ and $B$ such that \begin{equation}\label{qgen1}
	AT+BS=S'T_2.
	\end{equation} For such a $Y$ to be contraction, due to Proposition \ref{Dgenprop} it suffices to show that $A$, $B$ are contractions and \begin{equation}\label{qgen2}
	A=D_{B^*}CD_{T_2}
	\end{equation} for some contraction $C\in \mathcal{B}(\mathcal{H}_1,\mathcal{H}_2')$. We first construct $K$ from $\mathcal{H}_1$ to $\mathcal{H}_2'$ such that $A=KD_{T_2}$. Thus \eqref{qgen1} gives us 
	\begin{equation*} KD_{T_2}T+BS=S'T_2. \end{equation*} This implies 
	\begin{equation}\label{qgen3} T^*D_{T_2}K^*+S^*B^*=T_2^*S'^*. \end{equation} In order to find $K^*$ and $B^*$ satisfying \eqref{qgen3}, due to Theorem \ref{Dthm1}, it suffices to show that $$(T^*D_{T_2})(T^*D_{T_2})^*+S^*S\geq T_2^*S'^*S'T_2.$$ 
	Since $T_2^*T_2\leq I_{\mathcal{H}_1}$, $T^*T+S^*S=I_{\mathcal{H}_1}$ and $T'^*T'+S'^*S'=I_{\mathcal{H}_1'}$, we have 
	\begingroup
	\allowdisplaybreaks
	\begin{align*}
	(T^*D_{T_2})(T^*D_{T_2})^*+S^*S&=T^*(I-T_2^*T_2)T+S^*S\\
	&\geq T^*(I-T_2^*T_2)T+S^*S-(I_{\mathcal{H}_1}-T_2^*T_2)\\
	& = T_2^*T_2 - T^*T_2^*T_2T\\
	& = T_2^*T_2 - T_2^*T'^*T'T_2\\
	& = T_2^*(I_{\mathcal{H}_1'}-T'^*T')T_2\\
	& = T_2^*S'^*S'T_2.
	\end{align*}
	\endgroup
	Therefore, by Theorem \ref{Dthm1}, there exist operators $K:  \mathcal{H}_1 \to \mathcal{H}_2'$ and $B : \mathcal{H}_2 \to \mathcal{H}_2'$ satisfying \eqref{qgen3} and $KK^* + BB^*\leq I_{\mathcal{H}_2'}$. This implies that $B$ is a contraction and $KK^* \leq I_{\mathcal{H}_2'} - BB^*$. Again by applying Lemma \ref{Dlemma}, there exists a contraction $C$ from $\mathcal{H}_1$ to $\mathcal{H}_2'$ such that $K = D_{B^*}C$. Clearly $\|Y\|\geq \|T_2\|=1$. Since $T_2, B$ are contractions and $A=D_{B^*}CD_{T_2}$ for some contraction $C$, it follows from Proposition \ref{Dgenprop} that $\|Y\|\leq 1$.
	This completes the proof.
\end{proof}

Suppose $T$ is a contraction on $\mathcal{H}$ and $(V,\mathcal{K})$ is an isometric lift of $T$. Let $\mathcal{K}_1=\bigvee\limits_{n=0}^{\infty}V^n\mathcal{H}$ and $K_n=\bigvee\limits_{l=0}^nV^l\mathcal{H}$. Then $(V|_{\mathcal{K}_1}, \mathcal{K}_1)$ is the minimal isometric lift of $T$. Again, any two minimal isometric dilations of $T$ are unitarily equivalent. Thus, without loss of generality we may consider the Schaeffer's minimal isometric dilation $(V,\mathcal{K})$ of $T$, where $ \mathcal{K}=\mathcal{H}\oplus \mathcal{D}_{T}\oplus \mathcal{D}_{T}\oplus \cdots  $ and
\[
 V=\begin{pmatrix}
T&0&0&0&\cdots\\
D_T&0&0&0&\cdots\\
0&I_{\mathcal{D}_{T}}&0&0&\cdots\\
0&0&I_{\mathcal{D}_{T}}&0&\cdots\\
\vdots&\vdots&\vdots&\vdots&\ddots
\end{pmatrix}.
\]   
The following result is an analogue of Theorem \ref{Dmain} while we consider the minimal isometric lift of a contraction instead of a co-isometric extension. This is another main result of this paper.

\begin{thm}\label{main}
	Let $T_1$ be a nonzero contraction and $T_2$ be an operator on a Hilbert space $\mathcal{H}$. Suppose $T_1T_2=qT_2T_1$, where $q$ is a complex number satisfying $0 <|q|\leq \dfrac{1}{\|T_1\|}$. Let $V$ on $\mathcal{K}'$ be the minimal isometric dilation of $T_1$ and $V_q$ be an operator on $\mathcal{K}$ such that $qV_q$ is the minimal isometric dilation of $qT_1$. Then there exists a bounded linear operator $W:\mathcal{K} \rightarrow \mathcal{K}'$ such that 
	\begin{itemize}
		\item[(1)] $\mathcal{H}$ is invariant under $W^*$ and $W^*|_{\mathcal{H}}=T_2^*$;
		\item[(2)] $\|W\|=\|T_2\|$ and
		\item[(3)] $VW = qWV_q$.
	\end{itemize}
\end{thm}
\begin{proof}
	Since $\mathcal{K}'$ and $\mathcal{K}$ are the minimal isometric dilation spaces for $T_1$ and $qT_1$ respectively, without loss of generality we may assume that
	\[
	\mathcal{K}' =  \mathcal{H}\oplus \mathcal D_{T_1}\oplus \mathcal D_{T_1}\oplus \dots
	\;\; \text{	and } \; \;	\mathcal{K}=  \mathcal{H}\oplus \mathcal D_{qT_1}\oplus \mathcal D_{qT_1}\oplus \dots \; \;.
	\]
	Consider for all $n\in \mathbb{N}\cup \{0 \}$,
	\[ \mathcal{K}_n'=\mathcal{H}\oplus\underbrace{\mathcal{D}_{T_1}\oplus\cdots \oplus\mathcal{D}_{T_1}}_{\text{n copies}}\oplus \{ 0\}\oplus \dots \;\; \text{ and } \;\;
 \mathcal{K}_n=\mathcal{H}\oplus\underbrace{\mathcal{D}_{qT_1}\oplus\cdots \oplus\mathcal{D}_{qT_1}}_{\text{n copies}}\oplus \{ 0\}\oplus \dots \;\;.
 \] 
	It can be easily verified that $\mathcal{K}_n$ and $\mathcal{K}_n'$ are invariant subspaces for $(qV_q)^*$ and $V^*$ respectively. Hence define $V_n^*=V^*|_{\mathcal{K}_n'}$ and $(qV_q)_n^*=(qV_q)^*|_{\mathcal{K}_n}$.
	Since $\mathcal{K}_n'$ is an invariant subspace for $V_{n+1}^*$, with respect to the decomposition $\mathcal{K}_{n+1}'=\mathcal{K}_n'\oplus \mathcal{D}_{T_1}$ the operator $V_{n+1}$ has the block matrix form $V_{n+1}=\begin{pmatrix}
	V_n&0\\
	S_n'&0
	\end{pmatrix}$. Similarly, $\mathcal{K}_n$ is an invariant subspace for $(qV_q)_{n+1}^*$ hence, with respect to the decomposition $\mathcal{K}_{n+1}=\mathcal{K}_n\oplus \mathcal{D}_{qT_1}$ the operator $(qV_q)_{n+1}$ has the block matrix form $(qV_q)_{n+1}=\begin{pmatrix}
	(qV_q)_n&0\\
	S_n&0
	\end{pmatrix}$. Note that $V_n^*V_n+S_n'^*S_n'=I_{\mathcal{K}_n'}$ and $(qV_q)_n^*(qV_q)_n+S_n^*S_n=I_{\mathcal{K}_n}$. Now we find a sequence $\{W_n\}$ inductively using Theorem $\ref{qpart}$. We have $V_1=\begin{pmatrix}
	T_1&0\\
	D_{T_1}&0
	\end{pmatrix}$ and $(qV_q)_1=\begin{pmatrix}
	qT_1&0\\
	D_{qT_1}&0
	\end{pmatrix}$ 
	satisfying $T_2(qT_1)=T_1T_2$, $T_1^*T_1+D_{T_1}^2=I_{\mathcal{H}}$ and $(qT_1)^*(qT_1)+D_{qT_1}^2=I_{\mathcal{H}}$. Therefore, by Theorem \ref{qpart}, there exists $W_1=\begin{pmatrix}
	T_2&0\\
	A_1&B_1
	\end{pmatrix}$ from $\mathcal{K}_1$ to $\mathcal{K}_1'$ such that $W_1(qV_q)_1=V_1W_1$, $W_1^*|_{\mathcal{K}_0'}=T_2^*$ and $\|W_1\|=\|T_2\|$.
	
	Assume that for $1\leq m \leq n-1$, there exists $W_{m}=\begin{pmatrix}
	W_{m-1}&0\\
	A_{m}&B_{m}
	\end{pmatrix}$ from $\mathcal{K}_{m}$ to $\mathcal{K}_{m}'$ such that $W_{m}(qV_q)_{m}=V_{m}W_{m}$, $W_{m}^*|_{\mathcal{K}_{m-1}'}=W_{m-1}^*$ and $\|W_{m}\|=\|T_2\|$. Considering $V_n$ in place of $V'$ and $(qV_q)_n$ in place of $V$ in Theorem \ref{qpart}, there exists $W_{n}=\begin{pmatrix}
	W_{n-1}&0\\
	A_{n}&B_{n}
	\end{pmatrix}$ from $\mathcal{K}_{n}$ to $\mathcal{K}_{n}'$ such that $W_{n}(qV_q)_{n}=V_{n}W_{n}$, $W_{n}^*|_{\mathcal{K}_{n-1}'}=W_{n-1}^*$ and $\|W_{n}\|=\|T_2\|$. Hence by induction we have a sequence $\{W_n\}$ such that for each $n\in\mathbb{N}$, $W_{n}(qV_q)_{n}=V_{n}W_{n}$, $W_{n}^*|_{\mathcal{K}_{n-1}'}=W_{n-1}^*$ and $\|W_{n}\|=\|T_2\|$.
	We may consider $W_n^*$ as an operator from $\mathcal{K}'$ to $\mathcal{K}$ by defining $W_n^*(k)=0$ for $k\in \mathcal{K}'\ominus \mathcal{K}'_n$. Similarly $V_n^*$ and $(qV_q)^*_n$ can be thought of as operators on $\mathcal{K}'$ and $\mathcal{K}$ by defining it to be $0$ on $\mathcal{K}_n'^{\perp}$ and $\mathcal{K}_n^{\perp}$ respectively. Hence for all $n\in \mathbb{N}$, $(qV_q)_n^*W_n^*=W_n^*V_n^*$ from $\mathcal{K}'$ to $\mathcal{K}$. For each $x \in \bigcup\limits_{n=0}^{\infty}\mathcal{K}_n'$, which is dense in $\mathcal{K}'$, the sequence $\{W_n^*x\}$ is a cauchy sequence and since $\|W_n^*\|=\|T_2^*\|$, by uniform boundedness principle the sequence $\{W_n^*\}$ converges strongly, say to the operator $W^*\text{ from }\mathcal{K}'$ to $\mathcal{K}$. Clearly $\|W\|=\|T_2\|$ and $W^*|_{\mathcal{H}}=T_2^*$. By construction $\{V_n^*\}$ and $\{(qV_q)_n^*\}$ converge strongly to $V^*$ and $(qV_q)^*$ respectively. Hence $W(qV_q)=VW$. This completes the proof.
\end{proof}

\noindent One of the assumptions of the previous theorem was $TX=qXT$ for $0< |q|\leq\dfrac{1}{\|T\|}$. If we change the position of $q$, that is, if we consider $qTX=XT$ for $0< |q|\leq \dfrac{1}{\|T\|}$, then we have the following analogous result.

	\begin{thm}
		Let $T$ be a contraction on Hilbert space $\mathcal{H}$ and $X\in \mathcal{B}(\mathcal{H})$ be such that $qTX=XT$ for $0<|q|\leq \dfrac{1}{\|T\|}$. Suppose $(V,\mathcal{K}')$ and $(qV_q,\mathcal{K})$ are the minimal isometric dilations of $T$ and $qT$ respectively. Then there exists an operator $Y\in \mathcal{B}(\mathcal{K}',\mathcal{K})$ such that 
		\[ YV=qV_qY,\; Y^*(\mathcal{H})\subseteq \mathcal{H},\; Y^*|_{\mathcal{H}}=X^* \text{ and } \|X \|=\|Y\|  .\]
	\end{thm}
	\begin{proof}
		Note that $qTX=XT$ implies $\bar{q}X^*T^*=T^*X^*$. Since $V,qV_q$ are the minimal isometric dilations of $T, qT$ respectively, $V^*,(qV_q)^*$ are the minimal co-isometric extensions of $T^*, (qT)^*$ respectively. Then from Theorem \ref{Dmain}, there exists $Y^*$ from $\mathcal{K}$ to $\mathcal{K}'$ such that $V^*Y^*=Y^*\bar{q}V_q^*$, $Y^*|_{\mathcal{H}}=X^*$ and $\| X\|=\|Y\|$. That is, $YV=qV_qY$.
	\end{proof}
Here we have another consequence of the previous results.	
	
	\begin{lem}\label{comm.lift}
		Let $T_1\in \mathcal{B}(\mathcal{H}_1)$ and $T_2\in \mathcal{B}(\mathcal{H}_2)$ be contractions and $A\in \mathcal{B}(\mathcal{H}_1,\mathcal{H}_2)$ be such that $AT_1=T_2A$. Assume that $V_1 \in \mathcal{B}(\mathcal{K}_1)$ and $V_2 \in \mathcal{B}(\mathcal{K}_2)$ are minimal isometric dilations of $T_1 $ and $T_2$ respectively. Then there exists a bounded linear operator $B:\mathcal{K}_1'\to \mathcal{K}_2$ such that
		\begin{enumerate}
			\item $B^*$ maps $\mathcal{H}_2$ to $\mathcal{H}_1$ and $B^*|_{\mathcal{H}_2}=A^*$;
			\item $\|B\|=\|A\|$;
			\item $V_2B=BV_1$.
		\end{enumerate}
	\end{lem}
	\begin{proof}
		Clearly $AT_1=T_2A$ implies that $A^*T_2^*=T_1^*A^*$. Further, $V_i$ is the minimal isometric dilation of $T_i$ implies that $ V_i^*$ is the minimal co-isometric extension of $T_i^*$. Hence by Theorem \ref{douglasintertwining}, there exists $Y\in \mathcal{B}(\mathcal{K}_2,\mathcal{K}_1)$ such that $YV_2^*=V_1^*Y$, $Y\mathcal{H}_2\subseteq \mathcal{H}_1$, $Y|_{\mathcal{H}_2}=A^*$ and $\|Y\|=\|A\|$. Now by considering $B=Y^*$ we get the desired result. 
	\end{proof}
	
	The next result provides an analogue of Theorem \ref{main} in the unitary dilation setting. We consider here the minimal unitary dilation of a contraction $T$ and find a $q$-commutant lift of a $q$-commutant of $T$. This is a refinement of Proposition 2.6 of \cite{K.M.} and is another main result of this paper.	
	\begin{thm} \label{main112}
		Let $T_1$ be a nonzero contraction and $T_2$ be an operator on a Hilbert space $\mathcal{H}$. Suppose $T_1T_2=qT_2T_1$ with $0<|q|\leq \dfrac{1}{\|T_1\|}$. Let $U$ on $\mathcal{K}'$ be the minimal unitary dilation of $T_1$ and $U_q$ be an operator on $\mathcal{K}$ such that $qU_q$ is the minimal unitary dilation of $qT_1$. Then there exists a bounded linear operator $S:\mathcal{K} \rightarrow \mathcal{K}'$ such that 
		\begin{itemize}
			\item[(1)] $\|S\|=\|T_2\|$;
			\item[(2)] $US = qSU_q$ and
			\item[(3)] $T_1^nT_2 = P_{\mathcal{H}}U_q^nS|_{\mathcal{H}}$ and  $T_2T_1^n = P_{\mathcal{H}}SU^n|_{\mathcal{H}}$ for all $n\geq 0$.
		\end{itemize}
	\end{thm}
	\begin{proof}
		Let $(U_+,\mathcal{K}_+')$ and $\left((qU_q)_+,\mathcal{K}_+\right)$ be the minimal isometric dilations of $(T_1,\mathcal{H})$ and $(qT_1, \mathcal{H})$ respectively. Let $(U_q)_+=\dfrac{1}{q}(qU_q)_+$. Then by Theorem \ref{main}, there exists $S_+:\mathcal{K}_+ \rightarrow \mathcal{K}_+'$ such that 
		\begin{itemize}
			\item[(i)] $\mathcal{H}$ is invariant under $S_+^*$ and $S_+^*|_{\mathcal{H}}=T_2^*$;
			\item[(ii)] $\|S_+\| = \|T_2\|$ and 
			\item[(iii)] $U_+S_+ = qS_+(U_q)_+ = S_+(qU_q)_+$.
		\end{itemize}
		From (iii) we have that 
		$
		S_+^*U_+^* = (qU_q)_+^*S_+^*.
		$
		Since $(U^*, \mathcal{K}')\text{ and }((qU_q)^*, \mathcal{K})$ are the minimal isometric dilations of $(U_+^*, \mathcal{K}_+')\text{ and }((qU_q)_+^*, \mathcal{K}_+)$ respectively, by Lemma \ref{comm.lift} there exists $S^*: \mathcal{K}' \rightarrow \mathcal{K}$ such that
		\begin{itemize}
			\item[(a)] $S \text{ maps }\mathcal{K}_+ \text{ to } \mathcal{K}_+' \text{ and }S|_{\mathcal{K}_+} = S_+$ ,
			\item[(b)] $\|S\|=\|S_+\|=\|T_2\|$ ,
			\item[(c)] $S^*U^* = (qU_q)^*S^*$.
		\end{itemize}
		From (c) we have that 
		$
		US = qSU_q.
		$
		By (i) the block matrices of $S_+^*$ with respect to the decompositions $\mathcal{K}_+'= \mathcal{H} \oplus (\mathcal{K}_+'\ominus \mathcal{H})$ and $\mathcal{K}_{+}=\mathcal{H}\oplus (\mathcal{K}_+\ominus \mathcal{H})$ are of the form 
		$
		\begin{pmatrix}
		T_2^* & *\\
		0 & *
		\end{pmatrix}$, i.e.
		\[
		S_+=
		\begin{pmatrix}
		T_2 & 0\\
		* & *
		\end{pmatrix}:\mathcal{H}\oplus (\mathcal{K}_+\ominus \mathcal{H})\rightarrow \mathcal{H} \oplus (\mathcal{K}_+'\ominus \mathcal{H}).
		\]
		Again by (a) the block matrices of $S$ with respect to the decompositions $\mathcal{K}=\mathcal{K}_+\oplus \mathcal{K}_+^\perp$ and $\mathcal{K}'= \mathcal{K}_+' \oplus \mathcal{K}_+'^\perp$ are of the form 
		$
		\begin{pmatrix}
		S_+ & *\\
		0 & *
		\end{pmatrix}.
		$
		Therefore, the block matrices of $S$ with respect to the decompositions $\mathcal{K}=\mathcal{H}\oplus (\mathcal{K}_+\ominus \mathcal{H})\oplus \mathcal{K}_+^\perp$ and $\mathcal{K}'=\mathcal{H}\oplus (\mathcal{K}_+'\ominus \mathcal{H})\oplus \mathcal{K}_+'^\perp$ take the form 
		$\begin{pmatrix}
		T_2 & 0 & *\\
		* & * & *\\
		0 & 0 & *
		\end{pmatrix}$.
		Again, since $(U, \mathcal{K}')$ is the minimal unitary dilation of $(T_1, \mathcal{H})$, the block matrices of $U$ with respect to the decompositions $\mathcal{K}=\mathcal{H}\oplus (\mathcal{K}_+\ominus \mathcal{H})\oplus \mathcal{K}_+^\perp$ and $\mathcal{K}'=\mathcal{H}\oplus (\mathcal{K}_+'\ominus \mathcal{H})\oplus \mathcal{K}_+'^\perp$ have the form 
		$\begin{pmatrix}
		T_1 & 0 & *\\
		* & * & *\\
		0 & 0 & *
		\end{pmatrix}$.
		Similarly, the block matrix of $qU_q$ with respect to the decompositions $\mathcal{K}=\mathcal{H}\oplus (\mathcal{K}_+\ominus \mathcal{H})\oplus \mathcal{K}_+^\perp$ and $\mathcal{K}'=\mathcal{H}\oplus (\mathcal{K}_+'\ominus \mathcal{H})\oplus \mathcal{K}_+'^\perp$ would be 
		$\begin{pmatrix}
		qT_1 & 0 & *\\
		* & * & *\\
		0 & 0 & *
		\end{pmatrix}$.
		Therefore, we conclude that
		\[
		(qU_q)^nS  = \begin{pmatrix}
		(qT_1)^nT_2 & 0 & *\\
		* & * & *\\
		0 & 0 & *
		\end{pmatrix} \qquad \text{ for all } \;\; n \geq 0
		\] and consequently $T_1^nT_2 = P_{\mathcal{H}}U_q^nS|_{\mathcal{H}}$ for every non-negative integer $n$. Similarly, $T_2T_1^n = P_{\mathcal{H}}SU^n|_{\mathcal{H}}$ for all $n \geq 0$ and the proof is complete.
	\end{proof}

\end{document}